\documentclass[a4paper,11pt]{amsart}
\usepackage{geometry}
\usepackage{amsmath,amssymb,enumerate,bbm}
\usepackage{color}
\usepackage{graphicx}
\usepackage{url}

\newcommand{\id}{\mathrm{id}}

\newtheorem{thm}{Theorem}

\newtheorem{cor}[thm]{Corollary}

\newtheorem{lem}[thm]{Lemma}

\newtheorem{rmk}[thm]{Remark}

\begin{document}

\title[On Euler's decomposition formula for \lowercase{$q$}MZV\lowercase{s}]{On Euler's decomposition formula for \lowercase{$q$}MZV\lowercase{s}}

\author[J.~Castillo]{Jaime Castillo Medina}
\address{Univ.~de Val\`encia,
        	       Facultat de Ci\`encies Matem\`atiques,
	       C/Doctor Moliner, 50, 46100 Burjassot-Valencia, Spain.}       
         \email{jaicasme@alumni.uv.es}

\author[K.~Ebrahimi-Fard]{Kurusch Ebrahimi-Fard}
\address{Inst.~de Ciencias Matem\'aticas,
		C/Nicol\'as Cabrera, no.~13-15, 28049 Madrid, Spain.
		On leave from Univ.~de Haute Alsace, Mulhouse, France}
         \email{kurusch@icmat.es, kurusch.ebrahimi-fard@uha.fr}         
         \urladdr{www.icmat.es/kurusch}

\author[D.~Manchon]{Dominique Manchon}
\address{Univ. Blaise Pascal,
         C.N.R.S.-UMR 6620,
         63177 Aubi\`ere, France}       
         \email{manchon@math.univ-bpclermont.fr}
         \urladdr{http://math.univ-bpclermont.fr/~manchon/}

\date{11 September, 2013}         
%
%last update
%10/09/2013
%remark added: \label{future1}
%changes
%add DM as coauthor
%change: eq. (2) $\min(a,b-1-l)$
%

\maketitle

\tableofcontents
 
\vspace{0.3cm}

\begin{abstract}
In this short and elementary note we derive a $q$-generalization of Euler's decomposition formula for the \lowercase{$q$}-MZVs recently introduced by Y.~Ohno, J.~Okuda, and W.~Zudilin. This answers a question posed by these authors in \cite{OhOkZu}. 
\end{abstract}

\noindent

%%%%%%%%%%%%%%%%%%%%%%%%%%%%%%%%%%%%%%%%%%%

\section{Introduction}
\label{sect:intro}

Let $n_1,n_2, \ldots, n_k \in \mathbb{N}$ be positive natural numbers, with $n_1>1$. In this work the following $k$-fold iterated infinite series:
\begin{equation}
\label{qMZV-z}
	\bar{\mathfrak{z}}_q(n_1,\ldots,n_k)  	
	:=\sum_{m_1 > \dots > m_k > 0}
			\frac{q^{m_1}}{(1-q^{m_1})^{n_1} \cdots (1-q^{m_k})^{n_k}} 						
\end{equation}
are studied. The numbers $k > 0$ and $w := n_1+\cdots+n_k$ are defined as its length respectively weight. 

Following Rota's elementary description of the Jackson integral, we derive a formula involving binomial coefficients, which expresses for arbitrary $1 < a \leq b \in \mathbb{N}$ the product of two such series:
\begin{eqnarray}
\label{Z-Euler}
	\bar{\mathfrak{z}}_q(a)\bar{\mathfrak{z}}_q(b) &=&
	\sum_{l=0}^{a-1}\ \sum_{k=0}^{a-1-l}  
		(-1)^k {l+b-1 \choose b-1}{b \choose k} \bar{\mathfrak{z}}_q(b+l,a-l-k) 						\\
	& & \hspace{0.5cm}
	 	+ \sum_{l=0}^{b-1}\ \sum_{k=0}^{\min(a,b-1-l)}  
		(-1)^k {l+a-1 \choose a-1}{a \choose k} \bar{\mathfrak{z}}_q(a+l,b-l-k) 				\nonumber\\
	& & \hspace{1cm}
		- \sum_{k=1}^{a} \beta_{a-k} \bar{\mathfrak{z}}_q(a+b-k)
			+ \sum_{j=1}^{a-1} \alpha_{a+b-1-j} \delta \bar{\mathfrak{z}}_q(a+b-1-j).		\nonumber
\end{eqnarray}
Explicit expressions for the coefficients $\alpha_i$ and $\beta_j$, which depend on $a$ and $b$, are given further below. The last summand contains the derivation $\delta:=q\frac{d}{dq}$. 

Identity (\ref{Z-Euler}) answers a question posed in \cite{OhOkZu}, where the authors asked for such a decomposition formula -- in the differential algebra generated by (\ref{qMZV-z}), with respect to the derivation $\delta:=q\frac{d}{dq}$. In fact, the $q$-series defined in  (\ref{qMZV-z}) were studied in \cite{OhOkZu} in the light of possible $q$-analogs of multiple zeta values ($q$MZV). From this point of view (\ref{Z-Euler}) may be considered as a $q$-analog of Euler's decomposition formula.

Indeed, in the limit $q \nearrow 1$ the $q$-number $[m]_q:=(1-q^m)(1-q)^{-1}$ abridges $m \in \mathbb{N}$. In turn the above $k$-fold iterated sums multiplied by $(1-q)^w$ reduce to the corresponding classical multiple zeta values (MZVs) of length $k$ and weight $w$, defined for positive natural numbers $n_1, \ldots, n_k \in \mathbb{N}$,~$n_1>1$:
\begin{equation}
\label{clMZVs}
	\zeta(n_1,\dots,n_k):=\sum_{m_1  > \dots > m_k > 0}\frac{1}{ m_1^{n_1}\cdots m_k^{n_k}}.     
\end{equation}

MZVs of length one and two were already known to Euler, who showed, for instance, that the sum $\sum_{n>0} n^{-2}$ gives $\pi^2/6$.  A systematic study of MZVs began in the early 1990s with the seminal works of M.~E.~Hoffman \cite{Hoffman1} and D.~Zagier \cite{Zagier}. Since then these numbers have emerged in several mathematical areas including algebraic geometry, Lie group theory, advanced algebra, and combinatorics. Interestingly enough, MZVs also appear in theoretical physics, in particular in the context of perturbative quantum field theory. This lead to a fruitful interplay between mathematics and theoretical physics. See \cite{Hoffman2,Waldschmidt1,Waldschmidt2,Zudilin} for introductory reviews.  

Kontsevich \cite{Zagier} pointed out that the series (\ref{clMZVs}) has a nice and simple representation in terms of iterated Riemann integrals:   
\begin{equation} 
\label{ChenRepIntro}
	\zeta(n_1,\ldots,n_k) 
	= \idotsint\limits_{0\leq t_w \leq \cdots \leq t_1 \leq 1} 
			\frac{dt_1}{\tau_1(t_1)} \cdots \frac{dt_w}{\tau_w(t_w)},
\end{equation}
where $\tau_i(x) = 1-x$ if $i \in \{h_1,h_2,\ldots,h_k\}$, $h_j=n_1+n_2+\cdots+n_j$, and $\tau_i(x) = x$ otherwise.
  
The $\mathbb{Q}$-vector space spanned by (\ref{clMZVs}) forms an algebra, and there are two ways of writing a product of MZVs as a $\mathbb{Q}$-linear combination of MZVs. Indeed, the so-called double shuffle relations among MZVs arise from the interplay between the sum and the integral representations. The product of MZVs using the sum representation (\ref{clMZVs}) is usually called stuffle, or quasi-shuffle, product. The so-called shuffle product derives from the integral representation (\ref{ChenRepIntro}) of MZVs. The following simple example may elucidate this:
{\allowdisplaybreaks{
\begin{eqnarray}
\label{22-calc}	\zeta(2)\zeta(2)
	&=&\sum_{n>0}\frac{1}{n^2}\sum_{m>0}\frac{1}{m^2} 
	  =\sum_{m>n>0}\frac{1}{n^2m^2} + \sum_{n>m>0}\frac{1}{n^2m^2} + \sum_{n>0}\frac{1}{n^4}
	  =2\zeta(2,2) + \zeta(4)\\[0.2cm]
	&=& \Big(\int_0^1 \frac{dt_1}{t_1}\int_0^{t_1}\frac{dt_2}{1-t_2}\Big)\ 
			       		\Big(\int_0^1 \frac{dt_1}{t_1}\int_0^{t_1}\frac{dt_2}{1-t_2}\Big) \nonumber\\
	&=&4 \int_0^1 \frac{dt_1}{t_1} \int_0^{t_1} 
		\frac{dt_2}{t_2}\int_0^{t_2}\frac{dt_3}{1-t_3} \int_0^{t_3}\frac{dt_4}{1-t_4} \nonumber\\
	& & \qquad\     	+ 2 \int_0^1 
		\frac{dt_1}{t_1} \int_0^{t_1} \frac{dt_2}{1-t_2}\int_0^{t_2}\frac{dt_3}{t_3}\int_0^{t_3}\frac{dt_4}{1-t_4} 
		= 4\zeta(3,1) + 2\zeta(2,2). \nonumber
\end{eqnarray}}
The last equality is a special case of Euler's decomposition formula, which gives the general expression for natural numbers $1<a,b \in \mathbb{N}$:
\begin{equation}
\label{classical-euler}
	\zeta(a)\zeta(b)=
	\sum_{i=0}^{a-1} {i+b-1 \choose b-1} \zeta(b+i,a-i)
	+\sum_{j=0}^{b-1} {j+a-1 \choose a-1} \zeta(a+j,b-j).
\end{equation}
It can be proven using integration by parts, i.e., the shuffle product. Multiplying (\ref{Z-Euler}) on both sides by $(1-q)^{a+b}$, and then taking the limit $q \nearrow 1$ reduces it to (\ref{classical-euler}). 

Note that the third equality in line (\ref{22-calc}) is an example of the quasi-shuffle product for MZVs, which yields a rather different expression for the same product, known as Nielson reflexion formula \cite{Waldschmidt2}:
$$
	\zeta(a)\zeta(b)=\zeta(a,b)+\zeta(b,a) + \zeta(a+b).
$$

\vspace{0.3cm}

Finally, we notice that several $q$-analogs of MZVs have appeared in the literature. Especially the one proposed by Bradley in \cite{Bradley1} has been studied in detail. See e.g.~\cite{OhOkZu,Zhao,Zudilin}. We will show how Bradley's $q$-analog of MZVs can be related to linear combinations of series (\ref{qMZV-z}). The above $q$-analog of Euler's decomposition formula (\ref{Z-Euler}) transforms accordingly into the one described in \cite{Bradley2}.

\medskip

The paper is organized as follows. In the next section we recall a few facts about Jackson's $q$-analog of the  Riemann integral. We follow Rota's elementary description of Jackson's integral in terms of Rota--Baxter algebra. Then we introduce the above series  (\ref{qMZV-z}) as a possible $q$-analog of MZVs, and workout in detail an Euler-type decomposition formula for products of such series of length one in the differential algebra generated by (\ref{qMZV-z}), with respect to the derivation $\delta:=q\frac{d}{dq}$.  This answers a question posed by the authors in \cite{OhOkZu}. We compare with an already existing $q$-generalization of Euler's decomposition formula due to Bradley, and show that they are equivalent. Finally we mention how to use double $q$-shuffle relations to dispose of the $\delta$-terms in (\ref{qMZV-z}).

\vspace{0.5cm}
{\bf{Acknowledgements}}: The first author gratefully acknowledges support by the ICMAT Severo Ochoa Excellence Program. He would like to thank ICMAT for warm hospitality during his visit. The second author is supported by a Ram\'on y Cajal research grant from the Spanish government. KEF and DM were supported by the CNRS (GDR Renormalisation).\\[.3cm]

%%%%%%%%%%%%%%%%%%%%%%%%%%%%%%%%%%%%%%%%%%%

\section{The Jackson-Integral}
\label{sect:Jackson}

Recall that the Jackson integral:
\begin{eqnarray}
    	J[f](x) &:=& \int_{0}^{x} f(y) d_qy                		\nonumber\\
            	 &:=& (1-q)\:\sum_{n \ge 0} f(q^nx) q^nx  	\label{JI}
\end{eqnarray}
is the $q$-analog of the classical indefinite Riemann integral:
\begin{equation*}
      	R(f)(x):=\int_{0}^{x}f(y)\:dy.
\end{equation*}
The latter satisfies the classical integration by parts rule: 
\begin{equation}
\label{RB0}
	R(f)R(g)=R\big(fR(g)+ R(f)g\big), 
\end{equation}
which can be seen as dual to Leibniz' rule. 

We use Gian-Carlo Rota's elementary algebraic description of Jackson's integral  \cite{Rota1,Rota2} to find a $q$-analog of the last identity. For this we briefly recall the definition of Rota--Baxter algebra \cite{Baxter,Cartier,EFP,Rota1,Rota2}. It consists of an associative algebra $A$ over a field $\mathbb{K}$ of characteristic zero, which is equipped with a $\mathbb{K}$-linear map $T: A \to A$ satisfying the Rota--Baxter relation of weight $\theta \in \mathbb{K}$:
\begin{equation}
\label{RBRtheta}
	T(x)T(y)=T\big(T(x)y+xT(y) + \theta xy \big).
\end{equation} 
Note that the map $T':=-\lambda T$ for $\lambda \neq 0 \in \mathbb{K}$ satisfies the Rota--Baxter relation (\ref{RBRtheta}) for weight $\theta':=-\lambda\theta$. Integration by parts for the Riemann integral (\ref{RB0}) is an example for $\theta=0$.

\begin{rmk}\label{rmk:RBweight}{\rm{
In Baxter's original work \cite{Baxter} one finds identity (\ref{RBRtheta}) with the weight $\theta$ taking value in the algebra $A$.
}}  
\end{rmk} 

Let us define the multiplicative linear ($q$-dilation) operator:
\begin{equation*}
    	E_q[f](x):=f(qx).         
\end{equation*}
Its inverse is given by $E_{q^{-1}}$. The linear map:
\begin{equation*}
    	P_q[f]:=\sum_{n>0}E_q^{n}[f] = f(qx) + f(q^2x) + f(q^3x) + \cdots       
\end{equation*}
may be written formally as \cite{Cartier,Rota1}:
\begin{eqnarray}
     P_q 	&=& E_q(I + E_q + E^2_q + \cdots)       					\nonumber\\
       		&=& E_q \: \sum_{m \ge 0}E_q^{m} 
		  = \frac{E_q}{I-E_q} = \frac{I}{I-E_q} - I,                   		\label{Jack1}
\end{eqnarray}
where $I$ is the identity operator, that is, $I(f)=f$. The map $P_q$ satisfies the identity:
\begin{equation*}
	P_q[f]P_q[g] = P_q\big[P_q[f]g + fP_q[g] + fg \big], 
\end{equation*}
which makes it a Rota--Baxter map of scalar weight $\theta=1$. Indeed, for two functions $f,g$ we find that:
\begin{eqnarray}
     	P_q[f]P_q[g] &=& \sum_{n>0}E_q^{n}[f] \: \sum_{m>0}E_q^{m}[g]                                                      	\nonumber\\
              &=& \sum_{n>m>0}E_q^{n}[f]E_q^{m}[g] + \sum_{m>n>0} E_q^{n}[f]E_q^{m}[g] 
              			+ \sum_{m=n>0} E_q^{n}[f]E_q^{n}[g]  				\nonumber\\
%              &=& \sum_{n,m>0}E_q^{n+m}[f]E_q^{m}[g] + \sum_{m,n>0} E_q^{m}[f]E_q^{n+m}[g] 
%              			+ \sum_{n>0} E_q^{n}[fg]         					\nonumber\\
              &=& P_q\Big[ \sum_{n>0}E_q^{n}[f]\: g \Big]  
              			+ P_q\Big[ f \: \sum_{n>0}E_q^{n}[g] \Big] + P_q[fg]         \nonumber\\
              &=& P_q\big[P_q[f]g+ fP_q[g] + fg\big].                                                                            				\nonumber
\end{eqnarray}
Now let us define a multiplication operator $M_f $:
\begin{equation*}
	M_{f}[g](x):=(fg)(x)=f(x)g(x).          
\end{equation*}
Jackson's integral (\ref{JI}) may be expressed for fixed $q$, in terms of $P_q$ and $M_{\id}$:
\begin{eqnarray}
     	J[f](x) = (1-q)\sum_{n \ge 0} f(q^nx) q^nx             	
	          &=& (1-q)\sum_{n \ge 0} (f \id)(q^nx)          	\nonumber\\           
%	          &=& (1-q)\sum_{n \ge 0} (M_{\id}[f])\:(q^nx)        	\nonumber\\
             	  &=& (1-q)\sum_{n \ge 0} E^n_q[M_{\id}[f]](x)   	\nonumber\\
             	  &=& (1-q) (I+P_q)M_{\id}[f]\:(x) 
		    =: (1-q)\tilde{P}_qM_{\id}\:[f](x).                   	\label{qInt}
\end{eqnarray}
In the last equality, we defined the map:
\begin{equation*}
	\tilde{P}_q:= (I+P_q) =  \frac{I}{I-E_q} =  \sum_{m \ge 0}E_q^{m}, 
\end{equation*}
which satisfies the Rota--Baxter identity of weight $\theta=-1$:
\begin{equation}
\label{RBR1}
	\tilde{P}_q[f] \tilde{P}_q[g]  = \tilde{P}_q\big[\tilde{P}_q[f]g + f\tilde{P}_q[g] - fg\big].
\end{equation}
Jackson's integral $J=(1-q)\tilde{P}_qM_{\id}$ satisfies the relation:
\begin{equation}
\label{JacksonIntegral1}
    	J[f] J[g] + (1-q)JM_{\id}[f g] = J \big[J [f] g + f  J[g] \big].
\end{equation}
Using identity (\ref{RBR1}), we find indeed:
\begin{eqnarray}
	\lefteqn{J[f_1](x)\:J[f_2](x) = (1-q)\tilde{P}_q M_{\id}[f_1](x)  \: (1-q)\tilde{P}_q M_{\id}[f_2](x)}                	\nonumber \\
%                 &=& \int_{0}^{x} f_1 (y) d_qy \:   \int_{0}^{x} f_2 (z)d_qz\\[0.2cm]                        	\nonumber \\
%                 &=& (1-q)^2\sum_{n \ge 0} (f_1\: \id)(xq^n)\sum_{m \ge 0} (f_2\: \id)(xq^m)                  \nonumber \\
%                 &=& (1-q)^2\sum_{n,m \ge 0}f_1(xq^n)\:f_2(xq^m)\:x^2q^{n+m}                                 	\nonumber \\[0.2cm]
%
                 &=& (1-q)^2 \Big( \sum_{n,m \ge 0} \big( f_1(q^nx) f_2(q^{n+m}x) 
                 								 + f_1(q^{n+m}x) f_2(q^nx) \big)x^2q^{2n+m} 
								- \sum_{n \ge 0}f_1(q^nx)f_2(q^nx)q^{2n}x^2 \Big) 			\nonumber \\
%                 &=& \int_{0}^{x} f_1 (y)\int_{0}^{y} f_2 (s)d_qs d_qy +
%                                         \int_{0}^{x} f_2 (y)\int_{0}^{y} f_1 (s)d_qs d_qy 
%                                         -(1-q)\int_{0}^{x} f_1 (y) \: y \: f_2 (y)d_qy                                   		\nonumber\\
                 &=& (1-q)^2 \big(\tilde{P}_q \Big[ \tilde{P}_q[f_1\: \id] \: f_2\: \id \Big](x) 
                 		+ \tilde{P}_q \Big[f_1\: \id \: \tilde{P}_q[f_2 \: \id] \Big](x)  
                                        -  \tilde{P}_q[f_1 \: \id \: f_2\: \id](x)\Big),                       	     			\nonumber
\end{eqnarray}
which coincides with (\ref{JacksonIntegral1}).

\begin{rmk}{\rm{
\begin{enumerate}[i)]
\item
Observe that we can write (\ref{JacksonIntegral1}) as:
$$
	J[f]\: J[g] = J \big[J [f] \: g + f \: J[g] - (1-q)\id fg\big]. 
$$
Hence, $J$ may be interpreted as a Rota--Baxter map of algebra valued weight $\id_q:=-(1-q)\id$. 

\item
Using equation (\ref{JacksonIntegral1}) we arrive at the usual $q$-shuffle product for the Jackson integral \cite{Bradley1}. 
$$
	J[f]\:J[g] =J\big[ f\: J[g] \big]+  qJ\big[ J \big[E_q[f] \big]\:g \big].
$$

\item
We may compose $J$ with $M_{1/\id}$ --where it makes sense--, which leads to the modified $q$-integral:
\begin{eqnarray*}
    	\hat{J}[f](x) &:=& \int_{0}^{x} f(y) \frac{d_qy}{y}                		\nonumber\\
			 &=& (1-q)\:\sum_{n \ge 0} f(q^nx) = (1-q)\tilde{P}_q[f](x).             
\end{eqnarray*}
It satisfies a simpler identity of weight $\theta=-(1-q)$:
$$
	\hat{J}[f]\hat{J}[g] =\hat{J}\big[\hat{J}[f]g + f\hat{J}[g] - (1-q)fg\big].
$$
\end{enumerate}
}}
\end{rmk}

%%%%%%%%%%%%%%%%%%%%%%%%%%%%%%%%%%%%%%%%%%%

\section{\lowercase{$q$}-analogs of Multiple Zeta Values}
\label{sect:qMZV}

For positive natural numbers $n_1, \ldots, n_k \in \mathbb{N}$, $n_1>1$, multiple zeta values (MZVs) of length $k$ and weight $w:=n_1+n_2+\cdots+n_k$ are defined as $k$-fold iterated infinite series \cite{Hoffman2,Waldschmidt1,Zagier,Zudilin}:
\begin{eqnarray}
	\zeta(n_1,\dots,n_k)&:=&\sum_{m_1  > \dots > m_k > 0}
						\frac{1}{ m_1^{n_1}\cdots m_k^{n_k}} 				\label{MZVs}\\
				   &=&\idotsint\limits_{0\leq t_w \leq \cdots \leq t_1 \leq 1} 
				   		\frac{dt_1}{\tau_1(t_1)} \cdots \frac{dt_w}{\tau_w(t_w)}, 	\label{ChenRep}
\end{eqnarray}
where $\tau_i(u) = 1-u$ if $i \in \{h_1,h_2,\ldots,h_k\}$, $h_j=n_1+n_2+\cdots+n_j$, and $\tau_i(u) = u$ otherwise.

%%%%%%%%%%%%%%%%%%%%%%%%%%%%%%%%%%%%%%%%%%%

\subsection{Iterated Jackson integrals}
\label{ssect:JacksonMZV}

For further use, we define the functions $x:=1/\id$, $y:=1/(1-\id)$, and $\bar{y}:=\id/(1-\id)$, such that:
\begin{equation*}
	x(t)=\frac{1}{t}, \quad y(t)=\frac{1}{1-t}, \quad \bar{y}(t)=\frac{t}{1-t}.
\end{equation*}
Recall the common notation for $q$-numbers, $[m]_q:=\frac{1-q^m}{1-q}$. 

A natural way to define $q$-analogs of MZVs consists in replacing each Riemann integral in (\ref{ChenRep}) by its $q$-analog (\ref{qInt}). Therefore, we define for positive natural numbers $n_i \in \mathbb{N}$, $n_1>1$, $w:=n_1+\cdots+n_k$, the following map in terms of iterated Jackson integrals:  
\begin{equation}
\label{JacksonRep}
	Z(n_1,\dots,n_k)(t) := J\Big[ \rho_1 J\big[ \rho_2 \cdots J[\rho_w] \cdots \big] \Big](t),
\end{equation}
where $\rho_i(t) = y(t)$ if $i \in \{h_1,h_2,\ldots,h_k\}$, $h_j=n_1+n_2+\cdots+n_j$, and $\rho_i(t) = x(t)$ otherwise.

\smallskip

By evaluating (\ref{JacksonRep}) at $t=q$, we obtain a $q$-analog of MZVs, which was proposed in \cite{OhOkZu}. Indeed, since the constant function $1=M_{\id}(x)$, the iterated Jackson integrals in (\ref{JacksonRep}) reduce to iterations of the map $\tilde{P}_q =\sum_{n\ge 0}E_q^n$ applied to functions $\bar{y}$, that is, with $w:=n_1 + \cdots +n_k$ we obtain:
\begin{eqnarray}
     \mathfrak{z}_q(n_1,\ldots,n_k)	
     			&:=& Z(n_1,\dots,n_k)(q)  								\nonumber\\ 
     			&=&(1-q)^w\underbrace{\tilde{P}_q \: \big[\tilde{P}_q \: [\cdots \tilde{P}_q}_{n_1}\: [\bar{y}
                                     \cdots
                                     \underbrace{\tilde{P}_q \: [\tilde{P}_q\:[ \cdots \tilde{P}_q}_{n_k}\: [\bar{y}]]]
                                     \cdots \big](q)										\nonumber\\
                       	&=&	(1-q)^w \sum_{m_1  > \dots > m_k > 0}
                             	\frac{q^{m_1}}{(1-q^{m_1})^{n_1}\cdots (1-q^{m_k})^{n_k}}	\nonumber\\
         		&=& \sum_{m_1 > \dots > m_k > 0}
				\frac{q^{m_1 }}{[m_1]_q^{n_1}\cdots [m_k]_q^{n_k}}.  	\label{2qMZVs}
\end{eqnarray}
Indeed, for length $k>0$, and $w=\sum_{i=1}^k n_i$ we calculate:
\begin{eqnarray*}
     \lefteqn{(1-q)^w\underbrace{\tilde{P}_q \: \big[\tilde{P}_q \: [\cdots \tilde{P}_q}_{n_1}\: [\bar{y}
						\underbrace{\tilde{P}_q \: [\tilde{P}_q\:[ \cdots \tilde{P}_q}_{n_2}\: [\bar{y}
						\cdots
                                     		\underbrace{\tilde{P}_q \: [\tilde{P}_q\:[ \cdots \tilde{P}_q}_{n_k}\: [\bar{y}]
                                     		\cdots ]\big](q).}\nonumber\\
				&=&
     				(1-q)^{w-n_k}
     				\sum_{m > 0}\frac{1}{ [m_k]_q^{n_k}}	 
				\underbrace{\tilde{P}_q \: \big[\tilde{P}_q \: [\cdots \tilde{P}_q}_{n_1}\: [\bar{y}
						\underbrace{\tilde{P}_q \: [\tilde{P}_q\:[ \cdots \tilde{P}_q}_{n_2}\: [\bar{y}
						\cdots
                                     		\underbrace{\tilde{P}_q \: [\tilde{P}_q\:[ \cdots \tilde{P}_q}_{n_{k-1}}\: [\bar{y}\id^{m}]\cdots] \big] (q)\\
				&=&
				(1-q)^{n_1}
     				\sum_{m_2  > \dots > m_k > 0 \atop m>0 }\frac{1}{ [m_2]_q^{n_2}\cdots [m_k]_q^{n_k}}	 
				\underbrace{\tilde{P}_q \: \big[\tilde{P}_q \: [\cdots \tilde{P}_q}_{n_1}[\id^{m_2 + m}]\cdots] \big] (q)\\
				&=&(1-q)^{n_1}\sum_{m_1 > m_2  > \dots > m_k > 0 }
						\frac{1}{[m_2]_q^{n_2}\cdots [m_k]_q^{n_k}}  
				\tilde{P}_q \: \big[\tilde{P}_q \: [\cdots \tilde{P}_q[ \id^{m_1}]\cdots] \big] (q)\\
                             &=& \sum_{m_1  > \dots > m_k > 0}\frac{q^{m_1}}{[m_1]_q^{n_1}\cdots [m_k]_q^{n_k}}.     
\end{eqnarray*}
In the limit $q \nearrow 1$ the above $k$-fold iterated sums reduce to the corresponding classical MZV of length $k$ and weight $w$ defined in (\ref{MZVs}).

\medskip

We define now the main object of our study, i.e.,  the modified $q$MZV:
\begin{equation}
\label{modz}
    	\bar{\mathfrak{z}}_q(n_1,\ldots,n_k)
         :=\sum_{m_1  > \dots > m_k> 0}
			\frac{q^{m_1}}{(1-q^{m_1})^{n_1}\cdots (1-q^{m_k})^{n_k}},	
\end{equation}
for which $(1-q)^w \bar{\mathfrak{z}}_q(n_1,\ldots,n_k) = \mathfrak{z}_q(n_1,\ldots,n_k)$.

\begin{rmk}{\rm{
\begin{enumerate}[i)]
\item
The product of two such $q$MZVs of length one, using the sum representation (\ref{2qMZVs}), satisfies the quasi-shuffle like identity:
\begin{eqnarray}
    	\mathfrak{z}_q(n)\mathfrak{z}_q(m) &=& \mathfrak{z}_q(n,m) + \mathfrak{z}_q(m,n) + \mathfrak{z}_q(n+m) 	\label{qshz}\\
     		& & \qquad - (1-q)\big(\mathfrak{z}_q(n,m-1) + \mathfrak{z}_q(m,n-1) + \mathfrak{z}_q(n+m-1) \big).	\nonumber
\end{eqnarray}
Indeed, we see that:
\begin{eqnarray*}
	\mathfrak{z}_q(n)\mathfrak{z}_q(m) &=& (1-q)^{n+m} \Big( 
	\sum_{l_1>l_2>0} \frac{q^{l_1+l_2} }{(1-q^{l_1})^{n}(1-q^{l_2})^{m} }+
	\sum_{l_1>l_2>0} \frac{q^{l_1+l_2} }{(1-q^{l_1})^{m}(1-q^{l_2})^{n} }\\
	& & \hspace{3cm} + \sum_{l>0}  \frac{q^{2l}}{(1-q^{l})^{n+m}}\Big)\\
	 &=& (1-q)^{n+m} \Big( 
	\sum_{l_1>l_2>0} \frac{q^{l_1} - q^{l_1} (1-q^{l_2}) }{(1-q^{l_1})^{n}(1-q^{l_2})^{m} }+
	\sum_{l_1>l_2>0} \frac{q^{l_1} - q^{l_1} (1-q^{l_2}) }{(1-q^{l_1})^{m}(1-q^{l_2})^{n} }\\
	& & \hspace{3cm} + \sum_{l>0}  \frac{q^{l} - q^{l} (1-q^{l})}{(1-q^{l})^{n+m}}\Big)\\
	&=& \mathfrak{z}_q(n,m) + \mathfrak{z}_q(m,n) + \mathfrak{z}_q(n+m) 	\label{qshz}\\
     		& & \qquad - (1-q)\big(\mathfrak{z}_q(n,m-1) + \mathfrak{z}_q(m,n-1) + \mathfrak{z}_q(n+m-1) \big).
\end{eqnarray*}
In the limit $q \nearrow 1$ this product reduces to the usual quasi-shuffle for MZVs.

\item
Evaluating (\ref{JacksonRep}) at $t=1$ one obtains the $q$MZVs proposed by Schlesinger in \cite{Schles}:
\begin{eqnarray*}
     \zeta^S_q(n_1,\ldots,n_k) 	&:=& Z(n_1,\dots,n_k)(1)\\  
%     	&=& \underbrace{J[x\:J[x \cdots J[x\:}_{n_1-1}J[y
%                                  \cdots
%                \underbrace{J[x\: J[x \cdots J[x\:}_{n_l-1} J[y]] \cdots](1) 					\nonumber\\
%         &=&(1-q)^w\underbrace{\tilde{P}_q \: [\tilde{P}_q \: [\cdots \tilde{P}_q}_{n_1}\: [\bar{y}
%                                     \cdots
%               	\underbrace{\tilde{P}_q \: [\tilde{P}_q\:[ \cdots \tilde{P}_q}_{n_l}\: [\bar{y}]]
%                                     \cdots ](1)										\label{qMZV}\\
          &=& \sum_{m_1  > \dots > m_k > 0}\frac{1}{ [m_1]_q^{n_1}\cdots [m_k]_q^{n_k}}. 
\end{eqnarray*}
They satisfy the usual quasi-shuffle product:
$$
	\zeta^S_q(n)\zeta^S_q(m) = \zeta^S_q(n,m) + \zeta^S_q(m,n) + \zeta^S_q(n+m).
$$
\end{enumerate}
}}
\end{rmk}

%%%%%%%%%%%%%%%%%%%%%%%%%%%%%%%%%%%%%%%%%%%

\subsubsection{Another $q$-analog of MZVs} 

However, in the current literature, see e.g.~\cite{Bradley1,Bradley2}, another $q$-analog of MZVs is prevailing. It is defined using the weight one Rota--Baxter operator $P_q=E_q(I-E_q)^{-1}=\sum_{n >0}E_q^n$ defined in (\ref{Jack1}), together with the functions $x=1/\id$ and: 
$$
	y_k:=\frac{q^{-k}}{(1-q^{-k}\id)}, \quad k \in \mathbb{N}.
$$

Let us define again a map in terms of $k$-fold iterated Jackson integrals:
$$
	Z(n_1,\dots,n_k)(t) := J\Big[ \psi_1 J\big[ \psi_2 \cdots J[\psi_w] \cdots \big] \Big](t),
$$
where $n_1, \ldots, n_k$, $n_1>1$, are positive natural numbers, $w:=\sum_{i=1}^k n_i$, and $N_i:=n_1+n_2+\cdots+n_i$, $0<i<k+1$, such that $\psi_{N_i}(t) = y_i(t)$, and $\psi_i(t) = x(t)$ if $i \notin \{N_1,N_2,\ldots,N_k\}$. Then we find:
\begin{eqnarray}
     \zeta_q(n_1,\ldots,n_k) &:=&J\Big[ \psi_1 J\big[ \psi_2 \cdots J[\psi_w] \cdots \big] \Big](1)\nonumber\\
     				&=&(1-q)^w
     				\underbrace{P_q \: \big[P_q \: [\cdots P_q}_{n_1}\: [\bar{y}_1
				\underbrace{P_q \: [P_q\:[ \cdots P_q}_{n_2}\: [\bar{y}_{2}	
				 \cdots
                                     \underbrace{P_q \: [P_q\:[ \cdots P_q}_{n_k}\: [\bar{y}_{k}]]
                                     \cdots \big](1)												\nonumber\\
                             &=& \sum_{m_1  > \dots > m_k > 0}\frac{q^{(n_1-1)m_1 + \cdots + (n_k-1)m_k }}
                             								{[m_1]_q^{n_1}\cdots [m_k]_q^{n_k}},     
                             														\label{3qMZVs}
\end{eqnarray}
where $\bar{y}_i:=y_i \id$. By introducing another modified $q$MZV: 
\begin{eqnarray}
\label{BradleyModified}
      \bar{\zeta}_q(n_1,\ldots,n_k) =
      \sum_{m_1  > \dots > m_k > 0}\frac{q^{(n_1-1)m_1 + \cdots + (n_k-1)m_k }}
      							{(1-q^{m_1})^{n_1}\cdots (1-q^{m_k})^{n_k}},
\end{eqnarray}
we can write  $\zeta_q(n_1,\ldots,n_k) = (1-q)^w \bar{\zeta}_q(n_1,\ldots,n_k)$. Observe that we have the following quasi-shuffle product:
$$
	\zeta_q(n)\zeta_q(m) = \zeta_q(n,m) + \zeta_q(m,n) + \zeta_q(n+m) + (1-q)\zeta_q(n+m-1).
$$

In the following lemma we show how to relate the $q$MZVs in (\ref{2qMZVs}) and (\ref{3qMZVs}). 

\begin{lem}
\label{lem:AllqMZVs}
For $k>0$ natural numbers $n_1, \ldots, n_k$, $n_i>1$ for all $i$, and $w:=\sum_{i=1}^k n_i$:
\begin{equation} 
\label{one}
	\zeta_q(n_1,\ldots,n_k) = q^{w}\mathfrak{z}_{q^{-1}}(n_1,\ldots,n_k) 
			+  q^{w}\sum_{j=1}^{k-1} (1-q^{-1})^j\sum_{l_2+\cdots +l_{k} = j \atop  l_i \in \{0,1\},\ i=2,\ldots,k} 
							\mathfrak{z}_{q^{-1}}(n_1,n_2-l_2, \ldots,n_k-l_k).
\end{equation}
\end{lem} 

Before we prove the lemma, we recall that $\mathfrak{z}_{q^{-1}}$ is defined in terms of the map $\tilde{P}_{q^{-1}} = \sum_{n \ge 0} E^n_{q^{-1}}$:
\begin{eqnarray*}
	\mathfrak{z}_{q^{-1}}(n_1,\ldots,n_k) &:=& 
	(1-{q^{-1}})^w\underbrace{\tilde{P}_{q^{-1}} \: [\tilde{P}_{q^{-1}} \: [\cdots 
					\tilde{P}_{q^{-1}}}_{n_1}\: [\bar{y}
                                     \cdots
                                     \underbrace{\tilde{P}_{q^{-1}} \: [\tilde{P}_{q^{-1}}\:[ \cdots 
                                     	\tilde{P}_{q^{-1}}}_{n_k}\: [\bar{y}]]]
                                     \cdots ](q^{-1})\\
                                     				&=& 
				\sum_{m_1  > \dots > m_k > 0}
				\frac{q^{-m_1 }}{[m_1]_{q^{-1}}^{n_1}\cdots [m_k]_{q^{-1}}^{n_k}}.  	\label{q-invMZVs}	
\end{eqnarray*}
	
\begin{proof}
We re-write $\zeta_q(n_1,\ldots,n_k)$ as follows: 
\begin{eqnarray}
	\zeta_q(n_1,\ldots,n_k) &=&
	\sum_{m_1  > \dots > m_k > 0}\frac{q^{(n_1-1)m_1 + \cdots + (n_k-1)m_k }}
                             								{[m_1]_q^{n_1}\cdots [m_k]_q^{n_k}} \nonumber\\
				       &=&
	q^{w}\sum_{m_1  > \dots > m_k > 0}\frac{q^{-m_1 \cdots -m_k }}
                             								{[m_1]_{q^{-1}}^{n_1}\cdots [m_k]_{q^{-1}}^{n_k}}.	\label{put}			
\end{eqnarray}
Then we expand $q^{-m_1 \cdots -m_k}$ in terms of $q^{-m_1}$ and $(1-q^{m_i})$, $i=2,\ldots,k$. The simple first step:
$$
	q^{-m_1 \cdots -m_k} 	= q^{-m_1 \cdots -m_{k-1}}q^{-m_k}
					= q^{-m_1 \cdots -m_{k-1}}- q^{-m_1 \cdots -m_{k-1}}(1-q^{-m_k}), 
$$
yields inductively:
$$
	q^{-m_1 \cdots -m_k} 	
	=q^{-m_1} + q^{-m_1}\sum_{j=1}^{k-1} \sum_{l_2+\cdots +l_{k} = j \atop  l_i \in \{0,1\}, i=2,\ldots,k} 
	(-1)^{j}(1-q^{-m_{2}})^{l_2} \cdots (1-q^{-m_{k}})^{l_k}.
$$
Replacing the numerator in (\ref{put}) by the right hand side of the last equality implies (\ref{one}).
\end{proof}

For the modified $q$MZVs the expression is somewhat simpler, and we obtain:
\begin{cor}
\label{cor:AllqMZVs}
For $k>0$, $n_1, \ldots, n_k \in \mathbb{N}$, $n_i>1$ for all $i$, and $w:=\sum_{i=1}^k n_i$:
\begin{equation}
\label{two} 
	\bar{\zeta}_q(n_1,\ldots,n_k) = (-1)^{w}\bar{\mathfrak{z}}_{q^{-1}}(n_1,\ldots,n_k) 
			+  \sum_{j=1}^{k-1} \sum_{l_2+\cdots +l_{k} = j \atop  l_i \in \{0,1\},i=2,\ldots,k} 
							(-1)^{w-j}\bar{\mathfrak{z}}_{q^{-1}}(n_1,n_2-l_2,\ldots,n_k-l_k).
\end{equation}
\end{cor}

The general case of $k>0$ positive natural numbers $n_1, \ldots, n_k$ with only $n_1>1$, i.e., when some or all of the $n_2, \ldots, n_k$ may be one, is rather involved and left for further work. We only state the straightforward case of length two. 
 
\begin{cor}
\label{cor:length2case}
For $n>1$ we find:
\begin{equation} 
\label{two}
	\zeta_q(n,1) = q^{w}\mathfrak{z}_{q^{-1}}(n,1) +  q^{w} \sum_{l>0} \frac{(l-1)q^{-l}}{[l]_{q^{-1}}^{n}}
				= q^{w}\mathfrak{z}_{q^{-1}}(n,1) +  q^{w} \sum_{l>0} \frac{lq^{-l}}{[l]_{q^{-1}}^{n}} - q^{w}  \mathfrak{z}_{q^{-1}}(n).
\end{equation}
\end{cor}

\begin{proof}
This follows immediately from the identity:
$$
	\sum_{l>m>0} \frac{q^{-l}}{[l]_{q^{-1}}^{n}}=\sum_{l>0} \frac{(l-1)q^{-l}}{[l]_{q^{-1}}^{n}}.
$$
\end{proof}

At length $k=1$ the two modified $q$-analogs of MZVs differ only slightly:
$$
	\bar{\zeta}_q(a)=\sum_{n>0} \frac{q^{(a-1)n}}{(1-q^n)^a} 
				= \sum_{n>0} \frac{(-1)^a q^{-n}}{(1-q^{-n})^a}=(-1)^a\bar{\mathfrak{z}}_{q^{-1}}(a).
$$ 

\begin{rmk}{\rm{
Note that we can absorb the signs $(-1)^{w-j}$ in (\ref{two}) by defining another modified $q$MZV $\bar{\mathfrak{z}}'_{q^{-1}}(n_1,\ldots,n_k)$ in terms of the Rota--Baxter map  $\tilde{P}'_{q^{-1}} := -\tilde{P}_{q^{-1}}$, which is of weight $\theta=1$:
\begin{equation}
\label{NewmodqMZV}
	\bar{\mathfrak{z}}'_{q^{-1}}(n_1,\ldots,n_k) 
	=
	\underbrace{\tilde{P}'_{q^{-1}} \: [\tilde{P}'_{q^{-1}} \: [\cdots 
					\tilde{P}'_{q^{-1}}}_{n_1}\: [\bar{y}
                                     \cdots
                                     \underbrace{\tilde{P}'_{q^{-1}} \: [\tilde{P}'_{q^{-1}}\:[ \cdots 
                                     	\tilde{P}'_{q^{-1}}}_{n_k}\: [\bar{y}]]]
                                     \cdots ](q^{-1}),
\end{equation}
such that for $n_1, \ldots, n_k \in \mathbb{N}$, $n_i>1$ for all $i$:
$$
	\bar{\zeta}_q(n_1,\ldots,n_k) = \sum_{j=0}^{k-1} \sum_{l_2+\cdots +l_{k} = j \atop  l_i \in \{0,1\},i=2,\ldots,k} 
							\bar{\mathfrak{z}}'_{q^{-1}}(n_1,n_2-l_2,\ldots,n_k-l_k).
$$
This will affect the corresponding $q$-analog of Euler's decomposition formula for the $\bar{\zeta}_q(n_1,\ldots,n_k)$.
}}
\end{rmk}

%%%%%%%%%%%%%%%%%%%%%%%%%%%%%%%%%%%%%%%%%%%

\subsection{\lowercase{$q$}-analog of Euler's decomposition formula}
\label{ssect:qEuler}

In the following we will work mainly with the modified $q$MZVs $\bar{\mathfrak{z}}_{q}(n_1,\ldots,n_k)$ and $\bar{\mathfrak{z}}_{q^{-1}}(n_1,\ldots,n_k)$.  

\smallskip 

We now want to understand products of modified $q$-analogs of MZVs of length one. For this we use the Rota--Baxter algebra structure underlying the Jackson integral, implied by the map $\tilde{P}_q$. 

We introduce the notation $\tilde{P}_q^{(n)}$ for the following $n$-fold iterations:
$$
	\tilde{P}_q^{(n)}:=\tilde{P}_q \circ \tilde{P}_q \circ \cdots \circ \tilde{P}_q. 
$$
Then we may write the modified $q$MZV defined in (\ref{modz}) as follows:
\begin{equation*}
	 \bar{\mathfrak{z}}_q(n_1,\ldots,n_k) = 
	 \tilde{P}_q^{(n_1)}\big[\bar{y}\ \tilde{P}_q^{(n_2)}\big[\bar{y}\ \cdots \tilde{P}_q^{(n_k)}[\bar{y}]\big]\cdots\big](q).
\end{equation*}

We start with the simplest case. By invoking identity (\ref{RBR1}) for $\tilde{P}_q$ we calculate the product:
\begin{eqnarray}
\label{prod22}
     \bar{\mathfrak{z}}_q(2)\:\bar{\mathfrak{z}}_q(2) 
     			&=& (\tilde{P}_q\big[\tilde{P}_q[\bar{y}]\big]\: \tilde{P}_q\big[\tilde{P}_q[\bar{y}]\big])(q) 				\\
%     			&=& 2 \tilde{P}_q\big[\tilde{P}_q[\bar{y}] \tilde{P}_q[\tilde{P}_q[\bar{y}]]\big] 
%				      				- \tilde{P}_q\big[\tilde{P}_q[\bar{y}] \tilde{P}_q[\bar{y}] \big])(q) 				\nonumber\\
%                            	&=& 2 (\tilde{P}_q\Big[\tilde{P}_q\big[\bar{y} \tilde{P}_q[\tilde{P}_q[\bar{y}]]\big] + \tilde{P}_q\big[\tilde{P}_q[\bar{y}] 
%					\tilde{P}_q[\bar{y}]\big] - \tilde{P}_q\big[\bar{y}\tilde{P}_q[\bar{y}]\big]\Big])(q)					\nonumber\\
%		      	& & 		- \tilde{P}_q\Big[\tilde{P}_q\big[ 2\bar{y} \tilde{P}_q[\bar{y}]  - \bar{y}\bar{y}\big]\Big])(q)				\nonumber\\
%			&=& 2 (\tilde{P}_q\Big[\tilde{P}_q\big[\bar{y} \tilde{P}_q[\tilde{P}_q[\bar{y}]]\big] 
%				      			+ \tilde{P}_q\big[\tilde{P}_q\big[ 2\bar{y} \tilde{P}_q[\bar{y}]  - \bar{y}\bar{y}\big]\big]
%							- \tilde{P}_q\big[\bar{y}\tilde{P}_q[\bar{y}]\big]\Big])(q) 							\nonumber\\
%		      	& & 		- 2 \tilde{P}_q\Big[\tilde{P}_q\big[ \bar{y} \tilde{P}_q[\bar{y}] \big]\Big] 
%									+ \tilde{P}_q\Big[\tilde{P}_q\big[\bar{y}\bar{y}\big]\Big])(q)				\nonumber\\
			&=&  (2\tilde{P}_q\Big[\tilde{P}_q\big[\bar{y} \tilde{P}_q[\tilde{P}_q[\bar{y}]]\big] \Big]
				      			+ 4\tilde{P}_q\Big[\tilde{P}_q\big[\tilde{P}_q\big[ \bar{y} \tilde{P}_q[\bar{y}]\big]\big]\Big]
							- 4\tilde{P}_q\Big[\tilde{P}_q\big[\bar{y}\tilde{P}_q[\bar{y}]\big]\Big] 		\nonumber\\
			&&				- 2 \tilde{P}_q\Big[\tilde{P}_q\big[\tilde{P}_q\big[\bar{y}\bar{y}\big]\big]\Big] 
		      					+ \tilde{P}_q\Big[ \tilde{P}_q\big[\bar{y}\bar{y}\big]\Big])(q). 			\nonumber
\end{eqnarray}
This leads immediately to the identity:
\begin{eqnarray*}
	\bar{\mathfrak{z}}_q(2)\:\bar{\mathfrak{z}}_q(2) 
			&=& 2 \bar{\mathfrak{z}}_q(2,2) + 4 \bar{\mathfrak{z}}_q(3,1) - 4 \bar{\mathfrak{z}}_q(2,1)		\nonumber\\
			&&		(- 2 \tilde{P}_q\Big[\tilde{P}_q\big[\tilde{P}_q\big[\bar{y}\bar{y}\big]\big]\Big] 
		      							+ \tilde{P}_q\Big[ \tilde{P}_q\big[\bar{y}\bar{y}\big]\Big])(q). 	
\end{eqnarray*}
Observe that each time a $\tilde{P}_q$ disappears we encounter a minus sign. 

\begin{rmk}{\rm{
The above product (\ref{prod22}) with $\bar{\mathfrak{z}}_q(2)$ replaced by $\bar{\mathfrak{z}}'_q(2)$, defined in (\ref{NewmodqMZV}), leads to a similar identity for the latter, with $\tilde{P}_q$ replaced by $\tilde{P}'_q$, and minus signs turned into plus signs on the right hand side.
}}
\end{rmk} 

Now, from \cite{OhOkZu} we learn that:
\begin{equation}
\label{calc1}
	- 2 \tilde{P}_q\Big[\tilde{P}_q\big[\tilde{P}_q\big[\bar{y}\bar{y}\big]\big]\Big] (q) 
		      							+ \tilde{P}_q\Big[ \tilde{P}_q\big[\bar{y}\bar{y}\big]\Big](q) 	
	= 2  \bar{\mathfrak{z}}_q(3) -   \bar{\mathfrak{z}}_q(2) - \delta  \bar{\mathfrak{z}}_q(2), 								
\end{equation}
where $\delta:=q \frac{d}{dq}$. Indeed, we can calculate the left hand side: 
\begin{eqnarray}
	- 2 \tilde{P}_q\Big[\tilde{P}_q\big[\tilde{P}_q\big[\bar{y}\bar{y}\big]\big]\Big] (q) 
		      							+ \tilde{P}_q\Big[ \tilde{P}_q\big[\bar{y}\bar{y}\big]\Big](q)
%	&=& -2 \sum_{m,n>0}\tilde{P}_q\Big[\tilde{P}_q\big[\tilde{P}_q\big[ \id^{m+n}\big]\big]\Big](q)
%		+ \sum_{m,n>0}\tilde{P}_q\Big[\tilde{P}_q\big[ \id^{m+n}\big]\Big](q) 					       \nonumber\\
%	&=& -2 \sum_{m,n>0} \frac{q^{m+n}}{(1-q^{m+n})^3}	
%			+  \sum_{m,n>0} \frac{q^{m+n}}{(1-q^{m+n})^2}	\nonumber\\
	&=& - 2 \sum_{l>0} \frac{(l-1)q^{l}}{(1-q^{l})^3}	
			+  \sum_{l>0} \frac{(l-1)q^{l}}{(1-q^{l})^2}	\nonumber\\
	&=& -  \sum_{l>0} \frac{(l-1)q^{l}(1+q^l)}{(1-q^{l})^3},	\label{Zudilin1}		
\end{eqnarray}
which, following \cite{OhOkZu}, equals the right hand side in equality (\ref{calc1}). Hence, we obtain the $q$-shuffle product  of $\bar{\mathfrak{z}}_q(2)$ with itself in the differential algebra generated by the modified $q$MZV in (\ref{modz}) with respect to $\delta:=q \frac{d}{dq}$:
\begin{equation*}
	\bar{\mathfrak{z}}_q(2)\bar{\mathfrak{z}}_q(2) = 
		2 \bar{\mathfrak{z}}_q(2,2) + 4 \bar{\mathfrak{z}}_q(3,1) - 4 \bar{\mathfrak{z}}_q(2,1)
		+ 2  \bar{\mathfrak{z}}_q(3) -  \bar{\mathfrak{z}}_q(2) - \delta  \bar{\mathfrak{z}}_q(2).					
\end{equation*}

\begin{rmk}\label{future1}{\rm{
Going beyond length one, we calculated:
\begin{eqnarray*}
	\bar{\mathfrak{z}}_q(2,1)\bar{\mathfrak{z}}_q(2) &=&
	6\bar{\mathfrak{z}}_q(3,1,1) + 3\bar{\mathfrak{z}}_q(2,2,1) + \bar{\mathfrak{z}}_q(2,1,2)
	- 7 \bar{\mathfrak{z}}_q(2,1,1) \\
	& & \qquad\qquad\quad
	+ 4 \bar{\mathfrak{z}}_q(3,1) + \bar{\mathfrak{z}}_q(2,2) - 3 \bar{\mathfrak{z}}_q(2,1) - \delta\bar{\mathfrak{z}}_q(2,1),       	
\end{eqnarray*}
and 
\begin{eqnarray*}
	\bar{\mathfrak{z}}_q(3,1)\bar{\mathfrak{z}}_q(2) &=&
	9\bar{\mathfrak{z}}_q(4,1,1) + 4\bar{\mathfrak{z}}_q(3,2,1) + \bar{\mathfrak{z}}_q(3,1,2)+ \bar{\mathfrak{z}}_q(2,3,1)
	- 11 \bar{\mathfrak{z}}_q(3,1,1) \\
	& & - 2 \bar{\mathfrak{z}}_q(2,2,1) + \bar{\mathfrak{z}}_q(2,1,1)
	+ \bar{\mathfrak{z}}_q(3,2) - 5 \bar{\mathfrak{z}}_q(3,1) + 6 \bar{\mathfrak{z}}_q(4,1) - \delta\bar{\mathfrak{z}}_q(3,1).       	
\end{eqnarray*}
We will explore the full $\delta$-differential algebra structure generated by (\ref{modz}) in a forthcoming work.  
}}
\end{rmk}

We now derive a closed expression, i.e., Euler's decomposition formula for general products $a,b>1$:  
$$
	\bar{\mathfrak{z}}_q(a)\bar{\mathfrak{z}}_q(b) = \tilde{P}_q^{(a)}(\bar{y})(q) \tilde{P}_q^{(b)}(\bar{y})(q).
$$ 
In \cite{DP}, the authors gave an explicit formula for products $\tilde{P}_q^{(n)}(f)\tilde{P}_q^{(m)}(g)$. However, note that the expression displayed in theorem 3 of \cite{DP} seems to have problems. Following the idea of the proof in \cite{DP}, we calculated it again and arrived at the following formula, which is symmetric in $a$ and $b$: 
\begin{eqnarray}
\label{domain-sum}
	\tilde{P}_q^{(a)}(\bar{y})(q) \tilde{P}_q^{(b)}(\bar{y})(q)
	&=& 		\sum_{(i,j) \in D_1}(-1)^{a+b-i-j} {j-1 \choose i+j-b-1,j-a,a+b-i-j} \bar{\mathfrak{z}}_q(j,i) 		\\
	&&  + 	\sum_{(i,j) \in D_2}(-1)^{a+b-i-j} {j-1 \choose i+j-b,j-a,a+b-i-j-1} \bar{\mathfrak{z}}_q(j,i)		\nonumber\\
	&&    +	\sum_{(i,j) \in D_3} (-1)^{a+b-i-j}{j-1 \choose j-b,i+j-a-1,a+b-i-j}\bar{\mathfrak{z}}_q(j,i) 		\nonumber\\
	&&      +	\sum_{(i,j) \in D_4} (-1)^{a+b-i-j}{j-1 \choose j-b,i+j-a,a+b-i-j-1}\bar{\mathfrak{z}}_q(j,i) 		\nonumber\\
	&&        +	\sum_{(j) \in D_5} (-1)^{a+b-j}{j-1 \choose j-b,j-a,a+b-j-1}\tilde{P}^{(j)}_q(\bar{y}\bar{y})(q). 	\nonumber
\end{eqnarray}
The summation domains $D_i$, $i=1,\ldots,5$ are given by:
\begin{eqnarray*}
 	D_1&:=&\{ (i,j)\in \mathbb{N}_+ \times  \mathbb{N}_+ \ |\ 1 \leq i \leq b,\    a \leq j,\         b-i+1 \leq  j,\ j \leq  a+b-i     \}\\
	D_2&:=&\{ (i,j)\in \mathbb{N}_+ \times  \mathbb{N}_+ \ |\ 1 \leq i \leq b-1,\ a \leq j,\        b-i \leq  j,\      j \leq  a+b-i-1    \}\\
	D_3&:=&\{ (i,j)\in \mathbb{N}_+ \times  \mathbb{N}_+ \ |\ 1 \leq i \leq a,\     a-i+1 \leq j,\ b \leq  j,\        j \leq  a+b-i     \}\\
	D_4&:=&\{ (i,j)\in \mathbb{N}_+ \times  \mathbb{N}_+ \ |\ 1 \leq i \leq a-1,\  a-i \leq j,\     b \leq  j,\        j \leq  a+b-i-1 \}\\
	D_5&:=&\{j \in \mathbb{N} \ |\  a \leq j, b \leq j, j \leq a+b-1 \}.
\end{eqnarray*} 
Note the symmetry of $D_1,D_3$ and $D_2,D_4$ with respect to $a$ and $b$. The case $a=b=2$ can be easily verified to coincide with (\ref{prod22}).  

Let us assume $1< a \leq b$. We are particularly interested in the sum over domain $D_5$, which can be written as:  
$$
	\sum_{ j=0}^{a-1} (-1)^{a-j}{j+b-1 \choose j,j+b-a,a-j-1} \tilde{P}^{(j+b)}_q(\bar{y}\bar{y})(q).	
$$
This yields the $q$-series:
\begin{eqnarray*}
	z_{ab}&:=&\sum_{ j=0 \atop m,n>0}^{a-1} (-1)^{a-j}{j+b-1 \choose j,j+b-a,a-j-1} \frac{q^{m+n}}{(1-q^{m+n})^{j+b}}\\	
		  &=&\sum_{ j=0 \atop l>0}^{a-1} (-1)^{a-j}{j+b-1 \choose j,j+b-a,a-j-1}\frac{(l-1)q^{l}}{(1-q^{l})^{j+b}}.	
\end{eqnarray*} 
It can be rearranged into the handy expression:
\begin{equation}
\label{Problem}
	z_{ab}=\sum_{l>0} \frac{(l-1)q^{l}}{(1-q^l)^{a+b-1}} 
	\Big( \sum_{ j=0}^{a-1} (-1)^{a-j}{j+b-1 \choose j,j+b-a,a-j-1} (1-q^l)^{a-1-j}\Big).	
\end{equation}

With the goal to express this sum in terms of a linear combination of $\mathfrak{z}_q(s)$ and derivations $\delta \mathfrak{z}_q(t)$, let us look in detail at the next two cases, i.e., $a=2, b=3$ and $a= b=3$. 

We find:
\begin{eqnarray*}
	z_{23}&=& \sum_{l>0} \frac{(l-1)q^{l}}{(1-q^l)^{4}} 
				\Big( \sum_{ j=0}^{1} (-1)^{2-j}{j+2 \choose j,j+1,1-j} (1-q^l)^{1-j}\Big)\\
		  &=& \sum_{l>0} \frac{-(l-1)q^{l} (1+2q^l)}{(1-q^l)^{4}}. 
\end{eqnarray*} 
Now we show that $z_{23}=- \delta\bar{\mathfrak{z}}_q(3) + 3\bar{\mathfrak{z}}_q(4) - 2\bar{\mathfrak{z}}_q(3).$ First we observe that in general for $1<t \in \mathbb{N}$:
\begin{eqnarray}
	 \delta\bar{\mathfrak{z}}_q(t) 	&=& q\frac{d}{dq} \sum_{l>0} \frac{q^{l}}{(1-q^l)^{t}}				\nonumber\\
%						&=&   \sum_{l>0} \frac{lq^{l}}{(1-q^l)^{a}} - \sum_{l>0} \frac{-q^{l}a(-lq^l)}{(1-q^l)^{a+1}} \nonumber\\
%						&=&   \sum_{l>0} \frac{lq^{l}(1-q^l)}{(1-q^l)^{a+1}} - \sum_{l>0} \frac{alq^{2l}}{(1-q^l)^{a+1}}\nonumber\\
%						&=&   \sum_{l>0} \frac{lq^{l}(1-q^l) + alq^{2l}}{(1-q^l)^{a+1}}
						&=&   \sum_{l>0} \frac{lq^{l}(1+(t-1)q^l)}{(1-q^l)^{t+1}}. \label{deltaZ}
\end{eqnarray} 
Therefore, for $t=3$:
\begin{equation*}
	- \delta\bar{\mathfrak{z}}_q(3) = -  \sum_{l>0} \frac{lq^{l}(1+2q^l)}{(1-q^l)^{4}}.
\end{equation*} 
Next, we verify that:
\begin{eqnarray*}
	 3\bar{\mathfrak{z}}_q(4) - 2\bar{\mathfrak{z}}_q(3)  	
	 				&=&  \sum_{l>0} \frac{3q^{l}}{(1-q^l)^{4}}  
							 - \sum_{l>0} \frac{2q^{l}}{(1-q^l)^{3}}  \\
%					&=&  \sum_{l>0}q^{l} \frac{3 - 2 (1-q^l)}{(1-q^l)^{4}}\\
%					&=&  \sum_{l>0}q^{l} \frac{3 - 2 + 2q^l}{(1-q^l)^{4}}\\
%					&=&  \sum_{l>0} \frac{q^l + 2q^{2l}}{(1-q^l)^{4}} 
					&=&  \sum_{l>0} \frac{q^l(1+ 2q^{l})}{(1-q^l)^{4}}.      
\end{eqnarray*} 
Putting this together, gives:
\begin{eqnarray*}
	- \delta\bar{\mathfrak{z}}_q(3) + 3\bar{\mathfrak{z}}_q(4) - 2\bar{\mathfrak{z}}_q(3) 
%	&=& -  \sum_{l>0} \frac{lq^{l}(1+2q^l)}{(1-q^l)^{4}} + \sum_{l>0} \frac{q^l(1+ 2q^{l})}{(1-q^l)^{4}}\\ 
	&=& -  \sum_{l>0} \frac{(l-1)q^{l}(1+2q^l)}{(1-q^l)^{4}}.
\end{eqnarray*} 
And this yields:
\begin{eqnarray*}
	\bar{\mathfrak{z}}_q(2)\bar{\mathfrak{z}}_q(3)
	&=&   \tilde{P}_q^{(2)}(\bar{y})(q) \tilde{P}_q^{(3)}(\bar{y})(q)\\
%	&=& 	- 2 \bar{\mathfrak{z}}_q(3,1) + 3\bar{\mathfrak{z}}_q(4,1) - \bar{\mathfrak{z}}_q(2,2) + 2 \bar{\mathfrak{z}}_q(3,2)
%			+ \bar{\mathfrak{z}}_q(2,3)\\
%	& & + \bar{\mathfrak{z}}_q(2,1) - 2 \bar{\mathfrak{z}}_q(3,1) - \bar{\mathfrak{z}}_q(2,2)\\
%	& & - 2 \bar{\mathfrak{z}}_q(3,1) + 3 \bar{\mathfrak{z}}_q(4,1) + \bar{\mathfrak{z}}_q(3,2) - \bar{\mathfrak{z}}_q(3,1)\\	
%	&  &	- \delta\bar{\mathfrak{z}}_q(3) + 3\bar{\mathfrak{z}}_q(4) - 2\bar{\mathfrak{z}}_q(3)\\
	&=& 3 \bar{\mathfrak{z}}_q(3,2) + \bar{\mathfrak{z}}_q(2,3) + 6 \bar{\mathfrak{z}}_q(4,1)
			- 2 \bar{\mathfrak{z}}_q(2,2) - 7 \bar{\mathfrak{z}}_q(3,1) + \bar{\mathfrak{z}}_q(2,1)\\
	&  &	- \delta\bar{\mathfrak{z}}_q(3) + 3\bar{\mathfrak{z}}_q(4) - 2\bar{\mathfrak{z}}_q(3). 
\end{eqnarray*} 
We find the modified $q$-analogs of the usual shuffle product $\zeta(2)\zeta(3)= 3 \zeta(3,2) + \zeta(2,3) + 6 \zeta(4,1)$, as well as several lower weight terms and a derivation term. 

The next case $a=b=3$ leads to the $q$-series:
\begin{eqnarray*}
	z_{33}&=& \sum_{l>0} \frac{(l-1)q^{l}}{(1-q^l)^{5}} \Big( \sum_{ j=0}^{2} (-1)^{3-j}{j+2 \choose j,j,2-j} (1-q^l)^{2-j}\Big)\\
%		  &=& \sum_{l>0} \frac{(l-1)q^{l}}{(1-q^l)^{5}}\Big( - \frac{2!}{0!0!2!} (1-q^l)^2 
%									+ \frac{3!}{1!1!1!} (1-q^l) - \frac{4!}{2!2!0!} \Big)\\
%		  &=& \sum_{l>0} \frac{(l-1)q^{l}}{(1-q^l)^{5}}\Big( - (1-q^l)^2 + 6 (1-q^l) - 6 \Big)\\
%		  &=& \sum_{l>0} \frac{(l-1)q^{l}}{(1-q^l)^{5}}\Big( - 1 - 4q^l - q^{2l} \Big)
		  &=& - \sum_{l>0} \frac{(l-1)q^{l}}{(1-q^l)^{5}}\Big( 1 + 4q^l + q^{2l} \Big).
\end{eqnarray*} 
We now verify that:
$$
	z_{33}= - \frac{3}{2}\delta\bar{\mathfrak{z}}_q(4) + \frac{1}{2}\delta\bar{\mathfrak{z}}_q(3) 
	+ 6\bar{\mathfrak{z}}_q(5) - 6\bar{\mathfrak{z}}_q(4) + \bar{\mathfrak{z}}_q(3).
$$
Indeed, from (\ref{deltaZ}) follows that:
\begin{eqnarray*}
	- \delta\bar{\mathfrak{z}}_q(4) = -  \sum_{l>0} \frac{lq^{l}(1+3q^l)}{(1-q^l)^{5}},
\end{eqnarray*} 
and therefore:
\begin{eqnarray*}
	 \frac{3}{2}\delta\bar{\mathfrak{z}}_q(4) - \frac{1}{2}\delta\bar{\mathfrak{z}}_q(3) &= &
	  \frac{3}{2}\sum_{l>0} \frac{lq^{l}(1+3q^l)}{(1-q^l)^{5}} - \frac{1}{2}  \sum_{l>0} \frac{lq^{l}(1+2q^l)(1-q^l)}{(1-q^l)^{5}}\\
	 &=& \sum_{l>0} \frac{lq^{l}(\frac{3}{2}+\frac{9}{2}q^l - \frac{1}{2}  - \frac{1}{2} q^l + q^{2l})}{(1-q^l)^{5}} 
	 = \sum_{l>0} \frac{lq^{l}(1 +4q^l + q^{2l})}{(1-q^l)^{5}}.
\end{eqnarray*} 
Putting this together, we find:
\begin{eqnarray*}
	 6\bar{\mathfrak{z}}_q(5) - 6\bar{\mathfrak{z}}_q(4) + \bar{\mathfrak{z}}_q(3)	
%	 				&=&  \sum_{l>0} q^{l}\frac{6}{(1-q^l)^{5}}  
%							 + \sum_{l>0} q^{l}\frac{-6 + 6q^{l}}{(1-q^l)^{5}} 
%							 + \sum_{l>0}q^{l} \frac{(1-q^{l})^2}{(1-q^l)^{5}} \\
%					&=&  \sum_{l>0} q^{l}\frac{6 - 6 + 6q^l + 1 - 2q^l + q^{2l}}{(1-q^l)^{5}}\\
					&=&  \sum_{l>0} \frac{q^{l}(1 + 4q^l + q^{2l})}{(1-q^l)^{5}}.      
\end{eqnarray*} 

After these calculations, we turn to the general case. In the following theorem we show for $1<a \leq b$, that $z_{ab}$ in (\ref{Problem}) can be expressed as a linear combination of: 
$$
	\bar{\mathfrak{z}}_q(s) = \sum_{l>0} \frac{q^l}{(1-q^l)^s} 
	\qquad\ 
	\delta\bar{\mathfrak{z}}_q(t) = \sum_{l>0}\frac{lq^l(1+(t-1)q^l)}{(1-q^l)^{t+1}}.  
$$
In a subsequent corollary to this theorem, we answer a question stated in \cite{OhOkZu}, asking for a $q$-analog of Euler's decomposition formula for products $\bar{\mathfrak{z}}_q(a)\bar{\mathfrak{z}}_q(b)$, $1<a,b \in \mathbb{N}$, in the $\delta$-differential algebra generated by the $q$MZVs defined in (\ref{modz}).

\medskip 

\begin{thm}
\label{thm:product1}
Let $1 < a \leq b  \in \mathbb{N}$.  
\begin{eqnarray} 
	z_{ab} &:=&\sum_{l>0} \frac{(l-1)q^{l}}{(1-q^l)^{a+b-1}} \Big( 
				\sum_{ j=0}^{a-1} (-1)^{a-j}{j+b-1 \choose j,j+b-a,a-j-1} (1-q^l)^{a-1-j}\Big) \nonumber \\
		   &=& \sum_{k=0}^{a-2} \alpha_{k+b} \delta\bar{\mathfrak{z}}_q(k+b) 
		   		- \sum_{j=0}^{a-1} \beta_j \bar{\mathfrak{z}}_q(j+b). 				\label{result1}
\end{eqnarray}  	
with coefficients that depend on $a$ and $b$:
$$
	\beta_j=  (-1)^{a-j}{j+b-1 \choose j,j+b-a,a-j-1} 
	\qquad\
	\alpha_k= \sum_{j=b}^{k} \frac{(-1)^{a+b-j}}{1-j}{j-1 \choose j-b,j-a,a+b-j-1}.
$$
\end{thm}

\begin{proof}
We write $z_{ab}=z^1_{ab} - z^2_{ab}$, where:
$$
	z^1_{ab} := \sum_{l>0} \frac{lq^{l}}{(1-q^l)^{a+b-1}} \Big( 
				\sum_{ j=0}^{a-1} (-1)^{a-j}{j+b-1 \choose j,j+b-a,a-j-1} (1-q^l)^{a-1-j}\Big)
$$
and
$$
	z^2_{ab} := \sum_{l>0} \frac{q^{l}}{(1-q^l)^{a+b-1}} \Big( 
				\sum_{ j=0}^{a-1} (-1)^{a-j}{j+b-1 \choose j,j+b-a,a-j-1} (1-q^l)^{a-1-j}\Big).
$$
The goal is to show that:
$$
	z^1_{ab}=\sum_{k=0}^{a-2} \alpha_{k+b} \delta\bar{\mathfrak{z}}_q(k+b) 
\quad\ {\rm{and}}\quad\
	z^2_{ab} = \sum_{j=0}^{a-1} \beta_j \bar{\mathfrak{z}}_q(j+b).
$$
with coefficients that depend on $a$ and $b$. 

Let us start with the $z^2_{ab}$. Note that for $j \ge 0$:
$$
	\bar{\mathfrak{z}}_q(j+b) =\sum_{l>0} \frac{q^l(1-q^l)^{a -1-j}}{(1-q^l)^{a+b-1}}.  
$$ 
Then it follows straightforwardly that:
$$
	\sum_{j=0}^{a-1} \beta_j \bar{\mathfrak{z}}_q(j+b)
	=  \sum_{l>0} \frac{q^l}{(1-q^l)^{a+b-1}} \sum_{j=0}^{a-1} \beta_j(1-q^l)^{a -1-j},
$$
with: 
$$
	\beta_j:=  (-1)^{a-j}{j+b-1 \choose j,j+b-a,a-j-1}.
$$
Now we turn to $z^1_{ab}$. First, we observe that:
\begin{eqnarray*} 
	\delta\bar{\mathfrak{z}}_q(k) &=& \sum_{l>0}\frac{lq^l(1+(k-1)q^l)}{(1-q^l)^{k+1}}\\
%						   &=& \sum_{l>0}lq^l\frac{(1+(k-1)q^l)(1-q^l)^{a+b-2-k}}{(1-q^l)^{a+b-1}}\\
%						   &=& \sum_{l>0}lq^l\frac{(1- q^l)(1-q^l)^{a+b-2-k}
%						   					+ kq^l(1-q^l)^{a+b-2-k}}{(1-q^l)^{a+b-1}}\\
%						   &=& \sum_{l>0}lq^l\frac{(1-q^l)^{a+b-1-k}
%						   					- k(1-q^l)^{a+b-1-k} + k(1-q^l)^{a+b-2-k}}{(1-q^l)^{a+b-1}}\\
						   &=& \sum_{l>0}lq^l\frac{(1-k)(1-q^l)^{a+b-1-k}
						   					+ k(1-q^l)^{a+b-2-k}}{(1-q^l)^{a+b-1}}.
\end{eqnarray*} 
This yields:
\begin{eqnarray*} 
	\sum_{k=b}^{a+b-2} \alpha_{k} \delta\bar{\mathfrak{z}}_q(k)
		&=& \sum_{l>0}  \frac{lq^l}{(1-q^l)^{a+b-1}}  \sum_{k=b}^{a+b-2} \big(\alpha_{k}(1-k)(1-q^l)^{a+b-1-k}
						   					+ \alpha_{k}k(1-q^l)^{a+b-2-k}\big)\\
		&=& \sum_{l>0}  \frac{lq^l}{(1-q^l)^{a+b-1}}  \Big( \sum_{k=b}^{a+b-2} \alpha_{k}(1-k)(1-q^l)^{a+b-1-k}\\
		& & \hspace{4.5cm}
						   					+ \sum_{k=b+1}^{a+b-1} \alpha_{k-1}(k-1)(1-q^l)^{a+b-1-k}\Big).
\end{eqnarray*} 
The second sum on the right hand side can be written as:
\begin{eqnarray*} 
	\alpha_{b}(1-b)(1-q^l)^{a-1} &+& \sum_{k=b+1}^{a+b-2} \alpha_{k}(1-k)(1-q^l)^{a+b-1-k}\\
						  &+& \sum_{k=b+1}^{a+b-2} \alpha_{k-1}(k-1)(1-q^l)^{a+b-1-k}
						  				+ \alpha_{a+b-2}(a+b-2),
\end{eqnarray*} 
such that:
\begin{eqnarray*} 
	\sum_{k=b}^{a+b-2} \alpha_{k} \delta\bar{\mathfrak{z}}_q(k)
	&=& 
	 \sum_{l>0}  \frac{lq^l}{(1-q^l)^{a+b-1}}  \Big(
	 			\alpha_{b}(1-b)(1-q^l)^{a-1} + \alpha_{a+b-2}(a+b-2)\\
	& &	\hspace{4cm}		\sum_{k=b+1}^{a+b-2} (\alpha_{k}-\alpha_{k-1})(1-k)(1-q^l)^{a+b-1-k} \Big).
\end{eqnarray*} 
This is supposed to equal:
$$
	z^1_{ab} =  \sum_{l>0} \frac{lq^{l}}{(1-q^l)^{a+b-1}} \Big( 
				\sum_{k=b}^{a+b-1} c_{k-b} (1-q^l)^{a+b-1-k}\Big),
$$
with the coefficient: 
\begin{equation}
\label{coeff}
	c_{k-b}:=(-1)^{a+b-k}{k-1 \choose k-b,k-a,a+b-k-1}.
\end{equation} 
Comparison leads to the system:
\begin{eqnarray} 
	c_0 		&=&	\alpha_{b}(1-b)										\label{a1}\\
	c_{k-b} 	&=&	(\alpha_{k}-\alpha_{k-1})(1-k),\quad k=b+1,\ldots,a+b-2		\label{a2}\\
	c_{a-1}	&=&  \alpha_{a+b-2}(a+b-2).								\label{a3}
\end{eqnarray} 

Starting with (\ref{a1}) we find $\alpha_{b}=\frac{c_0}{1-b}$. Then (\ref{a2}) leads to the recurrence:
$$
	\alpha_{k}= \frac{c_{k-b}}{1-k} + \alpha_{k-1},
$$
which yields via induction the explicit expression for the coefficients:
$$
	\alpha_{k} = \sum_{j=b}^{k} \frac{c_{j-b}}{1-j}.
$$
\end{proof}

Note that for reasons of consistency, from (\ref{a3}) it should follow that:
$$
	\alpha_{a+b-2} = \frac{c_{a-1}}{a+b-2} = \sum_{j=b}^{a+b-2} \frac{c_{j-b}}{1-j}.
$$  
The last equality follows from the next 

\begin{lem}
For $1 < a \le b \in \mathbb{N}$:
\begin{eqnarray*} 
	\frac{1}{a+b-2}{a+b-2 \choose a-1,b-1,0} &= &\frac{1}{a+b-2}\frac{(a+b-2)!}{ (a-1)!(b-1)!}\\
								     &= &\sum_{j=0}^{a-2} \frac{(-1)^{a-j-1}}{j+b-1} {j+b-1 \choose j,j+b-a,a-j-1}.  	
\end{eqnarray*} 
\end{lem}

\begin{proof}
We will show that the coefficients (\ref{coeff}) satisfy $\sum_{j=0}^{a-1} \frac{-c_j}{j+b-1}=0$. Indeed:
\begin{eqnarray*} 
	\sum_{j=0}^{a-1} \frac{-c_j}{j+b-1} &=& \sum_{j=0}^{a-1}  \frac{(-1)^{a-j-1}}{j+b-1}{j+b-1 \choose j,j+b-a,a-1-j}\\ 
%	 						   &=& \sum_{j=0}^{a-1}  \frac{(-1)^{a-j-1}}{j+b-1} \frac{(j+b-1)!}{ j!(j+b-a)!(a-1-j)!}\\
%							   &=& \sum_{j=0}^{a-1}  \frac{(-1)^{a-j-1}}{(a-1)!} \frac{(j+b-2)!}{ (j+b-a)!} \frac{(a-1)!}{j!(a-1-j)!}\\
%							   &=& \frac{(-1)^{a-1}}{(a-1)!}\sum_{j=0}^{a-1} (-1)^{j} \frac{(j+b-2)!}{ (j+b-a)!}{a-1\choose j}\\
							   &=& \frac{(-1)^{a-1}}{(a-1)!}\sum_{j=0}^{a-1} (-1)^{j} \prod_{k=0}^{a-3} (j+b-2-k){a-1\choose j}.
\end{eqnarray*} 
Note that the polynomial: 
$$
	P_{ab}(j):=\frac{(j+b-2)!}{ (j+b-a)!}= \begin{cases} 1, a=2 \\
										  \prod_{k=0}^{a-3} (j+b-2-k), a>2
							     \end{cases} 
$$
is of degree $a-2 < a-1$. From \cite{MRuiz} we know that for $0 < i \leq n$ and $x \in \mathbb{R}$:
$$
	\sum_{j=0}^{a-1} (-1)^{j} (x-j)^{n-i}{a-1\choose j}=0,
$$
which implies that:
$$
	\sum_{j=0}^{a-1} (-1)^{j} \prod_{k=0}^{a-3} (j+b-2-k){a-1\choose j} = 0
$$
and therefore the lemma.
\end{proof}

Expression (\ref{domain-sum}) for the product of length one $q$MZVs defined in (\ref{2qMZVs}) is rather involved. Therefore, in the next corollary we derive an explicit and compact formula that deserves to be considered as a \lowercase{$q$}-analog of Euler's decomposition formula. In the next subsection we then indicate how to arrive at expression (\ref{Z-Euler}) in the introduction.
 
\begin{cor}
\label{cor:product1}
Let $1<a\leq b \in \mathbb{N}$.  
\begin{eqnarray} 
\label{better}
	\lefteqn{\bar{\mathfrak{z}}_q(a)\bar{\mathfrak{z}}_q(b)} \\
	&=& \sum_{j=a}^{a+b-1} \sum_{i=\max(b-j+1,1)}^{a+b-j} (-1)^{a+b-i-j}
					{j-1 \choose a-1}{a-1 \choose a+b-i-j}\bar{\mathfrak{z}}_q(j,i)			\nonumber \\
	& & \hspace{0.3cm} + \sum_{j=a}^{a+b-1} \sum_{i=\max(b-j,1)}^{a+b-j-1} (-1)^{a+b-i-j}
					{j-1 \choose a-1}{a-1 \choose a+b-i-j-1}\bar{\mathfrak{z}}_q(j,i)		 	\nonumber \\
	& & \hspace{0.5cm} + \sum_{j=b}^{a+b-1} \sum_{i=1}^{a+b-j} (-1)^{a+b-i-j}  
					{j-1 \choose b-1}{b-1 \choose a+b-i-j}\bar{\mathfrak{z}}_q(j,i)			\nonumber \\			
	& & \hspace{0.7cm} +  \sum_{j=b}^{a+b-1} \sum_{i=1}^{a+b-j-1} (-1)^{a+b-i-j}
					{j-1 \choose b-1}{b-1 \choose a+b-i-j-1}\bar{\mathfrak{z}}_q(j,i)			\nonumber \\   
	& & \hspace{1.2cm}+  \sum_{j=0}^{a-1}(-1)^{a-1} {j+b-1 \choose j,a-j-1,j+b-a} 
					\sum_{n>0}\frac{(n-1)q^n}{(1-q^n)^{j+b}}					\nonumber
\end{eqnarray} 	
\end{cor}

\begin{proof}
The proof follows by carefully unfolding the domains $D_1,\ldots, D_4$ defined below equation (\ref{domain-sum}). We briefly indicate it for:
$$
	D_1 = \{ (i,j)\in \mathbb{N}_+ \times  \mathbb{N}_+ \ |\ 1 \leq i \leq b,\    a \leq j,\ b-i+1 \leq  j,\ j \leq  a+b-i \} ,
$$
and show that:
\begin{eqnarray*} 
	\lefteqn{\sum_{(i,j) \in D_1}(-1)^{a+b-i-j} {j-1 \choose i+j-b-1,j-a,a+b-i-j} \bar{\mathfrak{z}}_q(j,i)} \\ 	
	& &\qquad\quad =
		\qquad\quad\  \sum_{j=a}^{a+b-1} \sum_{i=\max(b-j+1,1)}^{a+b-j} (-1)^{a+b-i-j}
					{j-1 \choose a-1}{a-1 \choose a+b-i-j}\bar{\mathfrak{z}}_q(j,i)	.
\end{eqnarray*} 
The combinatorial coefficient on each side match. Recall that we assume that $1<a \leq b$. The domain:
$$
	D_1 = \{ (i,j)\in \mathbb{N}_+ \times  \mathbb{N}_+ \ |\ 1 \leq i \leq b,\ \max(a,b-i+1) \leq j \leq  a+b-i \}.  
$$
By exchanging the order of summation in $i$ and $j$, and taking into account the min.~and max.~values for $j$, together with the resulting dominant constraint on $i$, we arrive at:
$$
	D_1 = \{ (i,j)\in \mathbb{N}_+ \times  \mathbb{N}_+ \ |\ a \leq j \leq a+b-1,\ \max(b-j+1,1) \leq i \leq  a+b - j \}.  
$$  
For the other domains similar arguments apply. 
\end{proof}

%%%%%%%%%%%%%%%%%%%%%%%%%%%%%%%%%%%%%%%%%%%

\subsection{Comparison with Bradley's \lowercase{$q$}-analog of Euler's decomposition formula}
\label{ssect:BradqEuler}

In \cite{Bradley2} Bradley derived a \lowercase{$q$}-analog of Euler's decomposition formula for the $q$MZVs defined in (\ref{3qMZVs}). Here we state the formula for the modified $q$MZV defined in (\ref{BradleyModified}). For natural numbers $1<a,b$:
\begin{eqnarray}
\label{Bradley-Euler}
	\bar{\zeta}_q(a)\bar\zeta_q(b) &=& 
		\sum_{n=0}^{a-1} \sum_{m=0}^{a-1-n} {n+b-1 \choose b-1}{b-1 \choose m} \bar{\zeta}_q(b+n,a-n-m) \\
	& & \hspace{1cm}
		+\sum_{n=0}^{b-1} \sum_{m=0}^{b-1-n} {n+a-1 \choose a-1}{a-1 \choose m} \bar{\zeta}_q(a+n,b-n-m)	\nonumber\\
	& & \hspace{2.5cm}
		 -\sum_{l=1}^{\min(a,b)} \frac{(a+b-1-l)!}{(a-l)!(b-l)!} \frac{1}{(l-1)!} 
		 \sum_{k>0}\frac{(k-1)q^{(a+b-1-l)k}}{(1-q^{k})^{a+b-l}}	.					\nonumber
\end{eqnarray}

Recall that $\bar{\mathfrak{z}}_{q^{-1}}(a)\ = (-1)^{a}\bar{\zeta}_q(a)$. Using (\ref{two}) in Corollary \ref{cor:AllqMZVs} we can write:
\begin{eqnarray*}
	\bar{\zeta}_q(b+n,a-n-m) &=& (-1)^{a+b-m}  \bar{\mathfrak{z}}_{q^{-1}}(b+n,a-n-m) \\
						& & \qquad\ +  (-1)^{a+b-m-1} \bar{\mathfrak{z}}_{q^{-1}}(b+n,a-n-m-1).   
\end{eqnarray*}
This inserted in (\ref{Bradley-Euler}), and recalling equation (\ref{two}) in Corollary (\ref{cor:length2case}), which gives: 
\begin{eqnarray*}
	\bar{\mathfrak{z}}_{q^{-1}}(b+k,0) &=& \sum_{l>0} \frac{(l-1)q^{-l}}{(1-q^{-l})^{b+k}}\\
							  &=& \sum_{l>0} \frac{lq^{-l}}{(1-q^{-l})^{b+k}} - \bar{\mathfrak{z}}_{q^{-1}}(b+k)
\end{eqnarray*}
we arrive after some careful adjustments of summation domains at the formula (\ref{Z-Euler}) for the product $\bar{\mathfrak{z}}_{q}(a)\bar{\mathfrak{z}}_{q}(b)$, displayed in the introduction. Hence, the link to Bradley's formula provides a simple way to deduce a more compact expression for the decomposition formula in (\ref{better}).

\begin{cor}
\label{cor:clEuler}
Multiplying (\ref{Z-Euler}) on both sides by $(1-q)^{a+b}$, and then taking the limit $q \nearrow 1$ results in the classical Euler decomposition formula:
\begin{equation}
\label{cl-euler-decomp}
	\zeta(a)\zeta(b)=
	\sum_{i=0}^{a-1} {i+b-1 \choose b-1} \zeta(b+i,a-i)
	+\sum_{j=0}^{b-1} {j+a-1 \choose a-1} \zeta(a+j,b-j).
\end{equation}
\end{cor}

\begin{proof}
First we multiply both sides of (\ref{Z-Euler}) by $(1-q)^{a+b}$. On the left hand side we obtain $\mathfrak{z}_q(a)\mathfrak{z}_q(b)$. On the right hand side in the first sum the modified $q$MZVs $\bar{\mathfrak{z}}_q(b+l,a-l-k)$ becomes the $q$MZV $(1-q)^{k}\mathfrak{z}_q(b+l,a-l-k)$. In the second sum we obtain the $q$MZV $(1-q)^{k}\mathfrak{z}_q(a+l,b-l-k)$. In the last two sums on the right hand side the modified $q$MZVs are replaced by $(1-q)^{k}\mathfrak{z}_q(a+b-k)$ and:
\begin{equation*}
	(1-q)^{a+b} \delta\bar{\mathfrak{z}}_q(a+b-1-j) = (1-q)^{j}\sum_{l>0} \frac{lq^{l}(1+(a+b-j-2)q^l)}{[l]_q^{a+b-j}}.
\end{equation*} 
respectively. Now, it is clear that the left had side reduces to $\zeta(a)\zeta(b)$ in the limit $q \nearrow 1$. Since $k,j>0$ the last two sums on the right hand side of (\ref{Z-Euler}) disappear in the limit $q \nearrow 1$. In the first two sums only the $k=0$ terms in the inner sums contribute on the right hand side. Which results in (\ref{cl-euler-decomp}). 
\end{proof}

Moreover, we may go backward and replace the last term in Bradley's formula (\ref{Bradley-Euler}).

\begin{cor}
\label{Bradley-Z-Euler}
For $1<a\le b$:
\begin{eqnarray*}
	\bar{\zeta}_q(a)\bar\zeta_q(b)  &=&
		\sum_{n=0}^{a-1} \sum_{m=0}^{a-1-n} {n+b-1 \choose b-1}{b-1 \choose m} \bar{\zeta}_q(b+n,a-n-m) \\
	& & \hspace{1cm}
		+\sum_{n=0}^{b-1} \sum_{m=0}^{b-1-n} {n+a-1 \choose a-1}{a-1 \choose m} \bar{\zeta}_q(a+n,b-n-m)	\nonumber\\
	& & 
		+ \sum_{k=1}^{a} (-1)^{k} \beta_{a-k}\bar{\zeta}_{q}(a+b-k)
			- \sum_{j=1}^{a-1} (-1)^{j+1} \alpha_{a+b-1-j} \delta \bar{\zeta}_{q}(a+b-1-j).	\nonumber
\end{eqnarray*}
\end{cor}
This gives a closed formula in the $\delta$-differential algebra generated by the modified $q$MZV defined in (\ref{BradleyModified}).

%%%%%%%%%%%%%%%%%%%%%%%%%%%%%%%%%%%%%%%%%%%

\subsection{Without $\delta\mathfrak{z}_q$-terms}
\label{ssect:noDeriv}

In the this final part we indicate how double $q$-shuffle relations may be used to eliminate the $\delta\mathfrak{z}_q$-terms. 

Let us first look at a the simplest example. From the quasi-$q$-shuffle formula (\ref{qshz}) for the modified $q$MZVs (\ref{modz}) we obtain:
$$    	\bar{\mathfrak{z}}_q(2)\bar{\mathfrak{z}}_q(2) = 2\bar{\mathfrak{z}}_q(2,2) + \bar{\mathfrak{z}}_q(4) 
		- 2 \bar{\mathfrak{z}}_q(2,1) - \bar{\mathfrak{z}}_q(3).
$$
On the other hand, from (\ref{Z-Euler}) we obtain:
$$
\bar{\mathfrak{z}}_q(2)\bar{\mathfrak{z}}_q(2) = 
		2 \bar{\mathfrak{z}}_q(2,2) + 4 \bar{\mathfrak{z}}_q(3,1) - 4 \bar{\mathfrak{z}}_q(2,1)
		+ 2  \bar{\mathfrak{z}}_q(3) -  \bar{\mathfrak{z}}_q(2) - \delta  \bar{\mathfrak{z}}_q(2).	
$$
This yields:
$$
	\delta  \bar{\mathfrak{z}}_q(2) = 
	  4 \bar{\mathfrak{z}}_q(3,1) - 2 \bar{\mathfrak{z}}_q(2,1)
		-  \bar{\mathfrak{z}}_q(2)
		+ 3 \bar{\mathfrak{z}}_q(3) 
	 	- \bar{\mathfrak{z}}_q(4) 
$$
Apparently this differs from the expression found in \cite{OhOkZu}.   
 
In general, from (\ref{qshz}) we find for $a=2$ and $b>2$:
\begin{eqnarray*}
    	\bar{\mathfrak{z}}_q(2)\bar{\mathfrak{z}}_q(b) &=& \bar{\mathfrak{z}}_q(2,b) + \bar{\mathfrak{z}}_q(b,2) 
											+ \bar{\mathfrak{z}}_q(2+b)\\
     		& & \qquad - \bar{\mathfrak{z}}_q(2,b-1) - \bar{\mathfrak{z}}_q(b,1) - \bar{\mathfrak{z}}_q(b+1).	\nonumber
\end{eqnarray*} 

From (\ref{Z-Euler}) for $a=2$ and $b>2$ we obtain:
\begin{eqnarray*}
	\bar{\mathfrak{z}}_q(2)\bar{\mathfrak{z}}_q(b) &=&
	\bar{\mathfrak{z}}_q(b,2) + \bar{\mathfrak{z}}_q(2,b) - b\bar{\mathfrak{z}}_q(b,1) + 2b\bar{\mathfrak{z}}_q(b+1,1)
	- 2 \bar{\mathfrak{z}}_q(2,b-1) + \bar{\mathfrak{z}}_q(2,b-2)	\\
		& &	
	+b\bar{\mathfrak{z}}_q(1+b) - (b-1) \bar{\mathfrak{z}}_q(b) - \delta \bar{\mathfrak{z}}_q(b)\\
	& & \hspace{0.5cm}
	 	+\sum_{l=1}^{b-2}\ \sum_{k=0}^{\min(2,b-1-l)}  
		(-1)^k {l+1 \choose 1}{2 \choose k} \bar{\mathfrak{z}}_q(2+l,b-l-k) 		
\end{eqnarray*}
Which yields for $b>2$:
\begin{eqnarray*}
	 \delta \bar{\mathfrak{z}}_q(b) &=& 
	 - (b-1)\bar{\mathfrak{z}}_q(b,1) + 2b\bar{\mathfrak{z}}_q(b+1,1)
	-  \bar{\mathfrak{z}}_q(2,b-1) + \bar{\mathfrak{z}}_q(2,b-2)\\
	& &
	+(b+1)\bar{\mathfrak{z}}_q(1+b) - \bar{\mathfrak{z}}_q(2+b) - (b-1) \bar{\mathfrak{z}}_q(b)\\
	& &
	 +\sum_{l=1}^{b-2}\ \sum_{k=0}^{\min(2,b-1-l)}  
		(-1)^k {l+1 \choose 1}{2 \choose k} \bar{\mathfrak{z}}_q(2+l,b-l-k) 
 \end{eqnarray*}
The algebraic structure underlying these double $q$-shuffle relations for the modified $q$MZVs (\ref{modz}) will be explored in a forthcoming work.

%%%%%%%%%%%%%%%%%%%%%%%%%%%%%%%%%%%%%%%%%%%
%%%%\newpage

\end{document}